\documentclass{amsart}
\usepackage{graphicx}
\usepackage{amssymb}
\newtheorem{thm}{Theorem}[section]

\newtheorem{lem}[thm]{Lemma}

\theoremstyle{definition}

\theoremstyle{question}

\theoremstyle{Conjecture}

\numberwithin{equation}{section}
\begin{document}

\title[on non-commuting sets and centralizers]
{on non-commuting sets and centralizers in infinite group}%
\author{Mohammad Zarrin}%

\address{Department of Mathematics, University of Kurdistan, P.O. Box: 416, Sanandaj, Iran}%
 \email{m.zarrin@uok.ac.ir, zarrin@ipm.ir}
 %\thanks{}%
%\subjclass{20D99, 20E07}%
%\keywords{centralizer; n-centralizer group; simple group.}%
%\date{}%
%\dedicatory{{\rm } }
%\commby{}%
% ----------------------------------------------------------------
\begin{abstract}
 A subset $X$ of a group $G$ is a set of pairwise non-commuting elements if $ab\neq ba$ for
any two distinct elements $a$ and $b$ in $X$. If $|X|\geq |Y|$ for
any other set of pairwise non-commuting elements $Y$ in $G$, then
$X$ is said to be a maximal subset of pairwise non-commuting
elements and the cardinality of such a subset is denoted by
$\omega(G)$. In this paper, among other thing, we prove that, for
each positive integer $n$,
there are only finitely many groups $G$, up to isoclinic, with $\omega(G)=n$ (with exactly $n$ centralizers).   
\end{abstract}
\maketitle
% ----------------------------------------------------------------

\section{\textbf{ Introduction and results}}

Let $G$ be a non-abelian group. We call a subset $X$ of $G$ a set
of pairwise non-commuting elements if $ab\neq ba$ for any two
distinct elements $a$ and $b$ in $X$. If $|X|\geq |Y|$ for any
other set of pairwise non-commuting elements $Y$ in $G$, then $X$
is said to be a maximal subset of pairwise non-commuting elements
and the cardinality of such a subset is called the clique number
of $G$ and and it is denoted by $\omega(G)$. By a famous result of
Neumann \cite{Neu} answering a question of ErdÄos, we know that
the finiteness of $\omega(G)$ in $G$ is equivalent to the
finiteness of the factor group $G/Z(G)$, where $Z(G)$ is the
center of $G$. Moreover,  Pyber \cite{Pyb} showed that $\omega(G)$
is also related to the index of the center of $G$. In fact, he
proved that there is some constant $c$ such that $[G:Z(G)]\leq
c^{\omega(G)}$. The clique number of groups was investigated by
many authors, for instance see \cite{Zar,Ber,Chi}.

It is easy to see that, if $H$ is an arbitrary abelian group and
$G$ is a group with $\omega(G)=n$, then  $\omega(G\times H)=n$.
Therefore, there can be infinitely many groups $K$ with
$\omega(K)=n$.  In this paper, by using a notion of isoclinic
groups (\cite{Hal}), first we show that the cardinality
 of maximal subset of
 pairwise non-commuting elements of any two isoclinic groups are that same (see Lemma \ref{l1}, below). Next, by this result, we show that,
 for each positive integer $n$,
there are only finitely many groups $G$, up to isoclinic, with
$\omega(G)=n$. Clearly, the relation isoclinic is an equivalence
relation on any family of groups and any two abelian groups are
isoclinic.\\

Our main results are.

\begin{thm}
Let  $n$ be a positive integer and $G$ be an arbitrary group such
that $\omega(G)=n$. Then
\begin{enumerate}
 \item There are only finitely many
groups $H$, up to isoclinic, with $\omega(H)=n$.\\ \item There
exists a finite group $K$ such that $K$ is isoclinic to $G$ and
$\omega(G)=\omega(K)$.
\end{enumerate}
\end{thm}

By this results, we give a sufficient condition for solvability by
its the cardinality of maximal subset of  pairwise non-commuting
elements.

 \begin{thm}
  Every arbitrary group $G$ with $\omega(G)\leq  20$ is solvable and this estimate is sharp.\\
\end{thm}

For any group $G$, let $\mathcal{C}(G)$ denote the set of
centralizers of $G$. We say that a group $G$ has $n~centralizers$
($G$ is a $\mathcal{C}_n$-group) if $|\mathcal{C}(G)| = n$.
Finally, we obtain similar results for groups with a finite number
$n$ of centralizers (see Lemma \ref{l2}, Theorem \ref{t32},
Theorem \ref{t34} and also Theorem \ref{t35}, below).

\section{\textbf{Pairwise non-commuting
elements}}

For prove the main results, we need the following Lemma.

Two groups $G$ and $H$ are said to be isoclinic if there are isomorphisms $\varphi:G/Z(G)\rightarrow H/Z(H)  $
 and $\phi: G' \rightarrow H'$ such that
$$\text{if}~ ~\varphi(g_1Z(G)) = h_1Z(H)$$$$
\text{and}~~ \varphi(g_2Z(G)) = h_2Z(H),$$$$
\text{then}~~ \phi([g_1, g_2]) = [h_1, h_2].$$
This concept is weaker than isomorphism and was introduced by P. Hall \cite{Hal} as a
structurally motivated classification for finite groups.
A stem group is defined as a group whose center is contained inside its derived
subgroup. It is known that every group is isoclinic to a stem group and if we restrict
to finite groups, a stem group has the minimum order among all groups isoclinic to it,
see \cite{Hal} for more details.

\begin{lem}\label{l1}
For every two isoclinic groups $G$ and $H$ we have
$\omega(G)=\omega(H)$.
\end{lem}
\begin{proof}
Suppose that $G$ and $H$ are two isoclinic groups.

 Therefore, according to P. Hall \cite{Hal}, there exist the commutator maps $$\alpha:G/Z(G)\times G/Z(G)\longrightarrow
 G',~(xZ(G),yZ(G))\mapsto ([x,y])$$ and $$
\alpha':H/Z(H)\times H/Z(H)\longrightarrow H', ~(xZ(H),yZ(H))\mapsto ([x,y])$$ and also isomorphisms $$\beta: G/Z(G)\longrightarrow
 H/Z(H),~\text{and}~~~~
 \gamma: G'\longrightarrow H'$$ such that $$\alpha'(\beta\times \beta)=\gamma(\alpha)$$
 where $$\beta\times \beta:G/Z(G)\times G/Z(G)\longrightarrow H/Z(H)\times H/Z(H).$$

 Now assume that the set $X=\{x_1,x_2,\dots,x_n\}$ is a a maximal subset of pairwise non-commuting
elements of $G$. It follows that $x_iZ(G)\neq x_jZ(G)$ for all
$1\leq i<j\leq n$. Therefore there exist $n$ elements $y_i\in
H\setminus Z(H)$ such that $\beta(x_iZ(G))=y_iZ(H)$. For completes
the proof it is enough to show that the set
$Y=\{y_1,y_2,\dots,y_n\}$ is a a subset of pairwise non-commuting
elements of $H$. Suppose, on the contrary, that there exist
$y_i,y_j\in H$ for some $1\leq i\neq j\leq n$, such that
$[y_i,y_j]=1$. Now, as mentioned above, we obtain that
$$\alpha'(\beta\times
\beta)((x_iZ(G),x_jZ(G)))=\gamma(\alpha)(x_iZ(G),x_jZ(G))$$ and
so $\alpha'(y_iZ(H),y_jZ(H))=\gamma([x_i,x_j])$ and so
$1=[y_i,y_j]=\gamma([x_i,x_j])$. It follows that $[x_i,x_j]=1$, a
contradiction. Thus $\omega(G)=|X|=|Y|\leq \omega(H)$ and so
$\omega(G)\leq \omega(H)$. Similarly, we get $\omega(H)\leq
\omega(G)$ and this completes the proof.
\end{proof}

By the above Lemma we prove Theorem 1.1.\\

 \noindent{\bf{Proof of Theorem 1.1.}}
(1)\; Assume that $G$ is a group with $\omega(G)=n$. According to
the Pyber \cite{Pyb}, there is some constant $c$ such that
$[G:Z(G)]\leq c^{\omega(G)}\leq f(n)$. Therefore, by Schur's
Theorem, the derived subgroup $G'$ is finite and also $|G'|\leq
f(n)^{2f(n)^3}$. Therefore there are finitely many isomorphism
types of $G/Z(G)$ and $G'$ which are bounded above by a function
of $n$. Therefore for every choice of $G/Z(G)$ and $G'$, there are
only finitely many commutator maps from $G/Z(G)\times G/Z(G)$ to
$G'$. It follows, in view of Lemma \ref{l1}, that $G$ is
determined by only
finitely isoclinism types.\\
(2)\; As $\omega(G)=n$, by Pyber \cite{Pyb}, $G$ is a
center-by-finite group. On the other hand, according to the main
Theorem of P. Hall (\cite{Hal}, p. 135), there exists a group $K$
such that $G$ is isoclinic to $K$ and $Z(K)\subseteq [K,K]=K'$. It
follows, as $G$ is isoclinic to $K$, that $K$ is center-by-finite
and so, according to Schur's Theorem, $K'$ is finite. Therefore
$Z(K)$ and $K/Z(K)$ are finite so $K$ is
finite and so Lemma \ref{l1} completes the proof.\\

Now we prove Theorem 1.2.\\

\noindent{\bf{Proof of Theorem 1.2.}} Assume that $G$ is a group
with $\omega(G)\leq 20$. Then according to Theorem 1.1, there
exists a finite group $K$ such that $G$ is isoclinic to $K$ and
$\omega(G)=\omega(K)$. Thus replacing $G$ by the factor group
$G/Z(G)$, it can be assumed without loss of generality that $G$ is
a finite group with $\omega(G)\leq 20$.  But in this case the
result follows from the main result of \cite{End} (note that the
alternating group of degree $5$, $A_5$ is a group with
$\omega(A_5)=21$ and so the estimate is sharp).

\section{\textbf{Groups with a finite number of centralizers} }

It is now appropriate to consider groups with a finite number $n$
of centralizers ($\mathcal{C}_n$-groups), since there exist the
interesting relations between centralizers and pairwise
non-commuting elements. For instance, as mentioned in the
introduction, the finiteness of $\omega(G)$ in $G$ is equivalent
to the finiteness of the factor group $G/Z(G)$. On the other hand,
because of centralizers are subgroups containing the center of the
group, the finiteness of the factor group $G/Z(G)$ follows that
$G$ has finite number of centralizers. Also if $G$ has finite
number of centralizers then it is easy to see that $\omega(G)$ is
finite. Therefore we can summarize the latter results in the
following theorem.
\begin{thm}
For any group G, the following statements are equivalent.
\begin{enumerate}
 \item $G$ has finitely many centralizers.
\item $G$  is a center-by-finite group. \item $G$ has finitely
many of pairwise non-commuting elements.
\end{enumerate}
\end{thm}

  It is
clear that a group is a $\mathcal{C}_1$-group if and only if it is
abelian. The class of $\mathcal{C}_n$-groups was introduced by
Belcastro and Sherman in~\cite{Be} and
investigated by many authors,  for instance see \cite{Ab, Ash, Zar1,Zar4, Zar3}.\\

As every group $G$ with a finite number of centralizers is
center-by-finite and so, by an argument similar to the one in the
proof of Lemma \ref{l1}, we will obtain the following result.

\begin{lem}\label{l2}
For every two isoclinic groups $G$ and $H$ we have
$|\mathcal{C}(G)|=|\mathcal{C}(H)|$.
\end{lem}
\begin{proof}
 Let $x$ be an element of $G$ and $\beta$ is the
isomorphism $\beta:G/Z(G)\longrightarrow H/Z(H)$. Therefore there
exists a subgroup $K$ of $H$ such that $\beta(C_G(x)/Z(G))=K/H$.
By an argument similar to the one in the proof of Lemma \ref{l1},
we show that there exist an element $y\in K$ such that $K=C_H(y)$
and $yZ(H)=\beta(xZ(G))$. Now as the isomorphism $\beta$ induces
a bijection between the subgroups of $G$ containing $Z(G)$ and the
subgroups of $H$ containing $Z(H)$ the result follows.
\end{proof}

Again, by an argument similar to the one in the proof of Theorems
1.1 we obtain the following result.
\begin{thm}\label{t32}
Let  $n$ be a positive integer and $G$ be an arbitrary
$\mathcal{C}_n$-group. Then
\begin{enumerate}
 \item There are only finitely many
groups $H$, up to isoclinic, with $|\mathcal{C}(H)|=n$; \item
There exists a finite group $K$ such that $K$ is isoclinic to $G$
and $|\mathcal{C}(G)|=|\mathcal{C}(K)|$.
\end{enumerate}
\end{thm}

For any group $G$, it is easy to see that if $x,y\in G$ and
$xy\neq yx$, then $C_G(x)\neq C_G(y)$, from which it follows
easily that $1+\omega(G)\leq |\mathcal{C}(G)|$ (note that
$C_G(e)=G$, where $e$ is the trivial element of $G$). Thus, by
using Theorem 2.1, we generalize Theorem A of \cite{Zar2}.

 \begin{thm}\label{t34}
  Every arbitrary group $G$ with $|\mathcal{C}(G)|\leq  20$ is solvable and this estimate is sharp.
\end{thm}
Finally, by using Theorem 3.3 (Case (2)), we generalize the main
results of \cite{Ab, Ash, Ash1, Be} for infinite groups, as follows:
\begin{thm}\label{t35}
  Let $G$ be an arbitrary $\mathcal{C}_n$-group. Then
  \begin{enumerate}
 \item $G/Z(G)\cong C_2\times C_2$ if and only if $n=4$.
\item $G/Z(G)\cong C_3\times C_3$  or  $S_3$ if and only if $n=5$.
\item $G/Z(G)\cong D_8,~A_4, C_2\times C_2\times C_2$ or $C_2\times C_2\times C_2\times C_2$ whenever  $n=6$.
\item $G/Z(G)\cong C_5\times C_5, D_{10}$ or  $\langle x, y| x^5=y^4=1, x^y=x^3 \rangle$ if and only if $n=7$.
\item $G/Z(G)\cong C_2\times C_2\times C_2, A_4$ or $D_{12}$ whenever  $n=8$.
\end{enumerate}
\end{thm}
\begin{proof}
For prove it is enough to note that there exists a finite
$\mathcal{C}_n$-group $K$ such that $K$ is isoclinic to $G$ so
$G/Z(G)\cong K/Z(K)$ and so the result follows from the main
results in \cite{Ab, Ash,Ash1,Be}.
\end{proof}

\end{document}